\documentclass[12pt]{amsart}
 \usepackage[dvips]{epsfig}
 \usepackage{comment}
 \usepackage{enumerate}
 \usepackage{amsgen, amstext,amsbsy,amsopn, amsthm, amsfonts,amssymb,amscd,amsmat
 h,euscript,enumerate,url,verbatim,calc,xypic}

 \usepackage{latexsym}
 \usepackage{graphics}
 \usepackage{color}
\usepackage[utf8]{inputenc}
\newcommand{\Pic}{\rm Pic\,}

\newcommand{\proset}{\,\mathrel{\lower 4pt\hbox{$\scriptscriptstyle/$}
\mkern -14mu\subseteq }\,} 
 
 \newtheorem{theorem}{Theorem}[section]
  \newtheorem{corollary}[theorem]{Corollary}
 \newtheorem{lemma}[theorem]{Lemma}
 \newtheorem{proposition}[theorem]{Proposition}

\newtheorem{remark}[theorem]{Remark}
 
 \newtheorem{definition}[theorem]{Definition}

 \newcommand{\prolim}[1]{\operatorname{``}\lim\limits_{#1}   \operatorname{''}}

\numberwithin{equation}{section}

\def\ker{\operatorname{ker}}

\usepackage{amsmath}
 \makeatother 
\begin{document}

\title{On the vanishing of Relative Negative K-theory}
\date{\today}
 \author{Vivek Sadhu}
 
 \address{School of Mathematics, Tata Institute of Fundamental Research, 1 Homi Bhaba Road, Colaba, Mumbai 400005, India}
 \address{Department of Mathematics, Indian Institute of Science Education and Research, Bhopal, Bhopal Bypass Road, Bhauri, Bhopal-462066, Madhya Pradesh, India}
 \email{sadhu@math.tifr.res.in, vsadhu@iiserb.ac.in, viveksadhu@gmail.com}
 

 \thanks{Author was supported by TIFR, Mumbai Postdoctoral Fellowship and IISER Bhopal grant INS/MATH/2017097}

\begin{abstract}
 In this article, we study the relative negative K-groups $K_{-n}(f)$ of a map $f: X \to S $ of schemes. We prove a relative version of Weibel's conjecture; i.e. if $f: X \to S$ is a smooth affine map of noetherian schemes with $\dim S=d$ then $K_{-n}(f)=0$ for $n> d+1$ and the natural map $K_{-n}(f) \to K_{-n}(f \times \mathbb{A}^{r})$ is an isomorphism for all $r>0$ and $n>d.$ We also prove a vanishing result for relative negative K-groups of a subintegral map.
\end{abstract}
\maketitle

\section{introduction}
In 1980, Weibel conjectured that for a $d$-dimensional noetherian scheme $X$, the negative $K$-groups should vanish after the dimension and the natural map $K_{-d}(X) \to K_{-d}(X \times \mathbb{A}^{r})$ for all $r>0$ should be an isomorphism; i.e. $X$ should be $K_{-d}$-regular (see Question 2.9 of \cite{Wei80}). Significant progress related to this conjecture has been made in the articles \cite{Cisinki}, \cite{CHSW}, \cite{GH}, \cite{Has}, \cite{KS}, \cite{SK}, \cite{Krishna}, \cite{wei duke} by various authors. Very recently, a complete answer is given in \cite{KST} by Kerz-Strunk-Tamme. 

 Let $f: X\to S$ be a morphism of schemes. By definition, the $n$-th relative K-group $K_{n}(f)$ is $\pi_{n}K(f)$, where $n \in \mathbb{Z}$ and $K(f)$ is the homotopy fiber of $K(S) \to K(X).$ Here and throughout, $K(X)$ denotes the non-connective Bass $K$-theory spectrum of the scheme $X$. Similarly, by replacing $K$ by $KH,$ we get the $n$-th relative homotopy $K$-group $KH_{n}(f).$ We say that $f: X \to S$ is $K_{n}$-regular if the natural map $K_{n}(f) \to K_{n}(f \times \mathbb{A}^{r})$ is an isomorphism for all $r>0.$ In this article, we consider Weibel's conjecture in the relative setting. More precisely, we are interested in investigating the condition on $f$ under which an analogous vanishing and regularity result holds for the relative negative $K$-groups. 

Firstly, we consider the case when $f$ is a smooth affine map. We discuss such a case in Section \ref{using KST}.  Using the technique of \cite{KS} and \cite{KST}, we prove the following 

\begin{theorem}\label{smooth affine}
 Let $f: X \to S$ be a smooth, affine map of noetherian schemes. Assume that $\dim S= d.$ Then $K_{-n}(f)=0$ for $n>d+1$ and $f$ is $K_{-n}$-regular for $n>d.$
\end{theorem}

However, we observe that the above theorem also holds for smooth, quasi-projective maps with a reduced base (see Theorem \ref{smooth quasiprojective}). But for a non-reduced base the result may fail (see Remarks \ref{regularity fails in general}, \ref{false in surface case}).

Secondly, we consider the case when $f$ is smooth, but may not be affine. In this situation, we are able to prove a vanishing result for relative negative homotopy $K$-groups assuming the resolution of singularities. We prove such a result in Section \ref{KH vanish}. In Section \ref{KH vanish}, all the schemes are defined over a field $k.$ 

For a scheme $X,$ let $Sch/X$ be the category of schemes of finite type over $X.$ Given a morphism $f: X \to S$ of schemes, consider the relative $K$-theory presheaf $ (U\to S) \mapsto K_{n}(f|_{U})$ on the category $Sch/S,$ where $f|_{U}$ is the map $X\times_S U \to U.$ Let $a$ denote the natural morphism from the cdh-site to the Zariski site on the category $Sch/S.$ Then $a_{cdh}\mathcal{K}_{n, zar}^{f}$ denotes the cdh-sheafification of the relative $K$-theory presheaf. Here $\mathcal{K}_{n, zar}^{f}$ is the associated Zariski sheaf. For simplicity, we write $\mathcal{K}_{n, cdh}^{f}$ instead of $a_{cdh}\mathcal{K}_{n, zar}^{f}.$ Now we state the main result of the section \ref{KH vanish}.
\begin{theorem}\label{Smooth,surjective}
Let $f: X\to S$ be a smooth map of noetherian schemes over a field $k$. Assume that the resolution of singularities holds over $k$ (see the section \ref{KH vanish} for the definition) and $\dim S=d.$ Then $KH_{-n}(f)=0$ for $n>d+1$ and $H_{cdh}^{d}(S, \mathcal{K}_{-1, cdh}^{f})\cong KH_{-d-1}(f).$ 
\end{theorem}


Next, we discuss the situation when the map $f: X \to S$ may not be smooth. In particular, we consider subintegral maps. In \cite{SW1}, the author and Weibel have shown that if $f: X=\rm{Spec}(B) \to S=\rm{Spec}(A)$ is a subintegral map (i.e. $A\hookrightarrow B$ is subintegral) then $K_{-n}(f)=0$ for $n>0$ (see Proposition 2.5 of \cite{SW1}). It has also been observed in \cite[Example 6.6]{SW1} that if $S$ is not affine then the above mentioned result may fail. For example, consider $S= \mathbb{P}_{k}^{1}$ and $X= \rm{Spec} (\mathcal{O}_{B})$ where $\mathcal{O}_{B}= \mathcal{O}_{S}\oplus \mathcal{O}(-2)$ with $\mathcal{O}(-2)$ being a square zero ideal. In this situation, $K_{-1}(f)\neq 0$. 
So it is natural to wonder what the groups $K_{-n}(f)$ are for subintegral morphisms with non affine base. This is answered in Section \ref{subint case} by proving the 
following theorem.

\begin{theorem}\label{K-dim}
 Let $f: X\to S$ be a subintegral morphism of noetherian schemes. Assume that $\dim S=d$. Then
 \begin{enumerate}
  \item $K_{-n}(f)=0$ for $n>d$,
  \item $H_{zar}^{d}(S, f_{*}\mathcal{O}^{\times}_X/\mathcal{O}^{\times}_{S})= K_{-d}(f)$,
  \item If $X$ and $S$ are $\mathbb{Q}$-schemes then $H_{et}^{d}(S, f_{*}\mathcal{O}^{\times}_X/\mathcal{O}^{\times}_{S})= K_{-d}(f).$
 \end{enumerate}
\end{theorem}
As a corollary, we obtain $K_{-n}(X)\cong K_{-n}(X_{\rm{sn}})$ for  $n> d$ and $K_{-d}(X) \to K_{-d}(X_{\rm{sn}})$ is surjective, where $X$ is a $d$-dimensional noetherian scheme and $X_{\rm{sn}}$ is the seminormalization of $X$ (see Corollary \ref{seminormalization}).

In Section \ref{relative regularity}, we prove a relative version of Vorst's regularity result that $K_{n}$-regularity implies $K_{n-1}$-regularity. More precisely, we prove 

\begin{theorem}\label{Analouge of Vorst}
 If $f: X=\rm{Spec}(B)\to S=\rm{Spec}(A)$ is $K_{n}$-regular then $f$ is $K_{n-1}$-regular.
\end{theorem}
 As a consequence, we show that if $f: A\hookrightarrow B$ is a non-trivial subintegral ring extension then $f$ cannot be $K_{n}$-regular(see Proposition \ref{reg for subintegral}).  
 \medskip
 

{\bf Acknowledgements:} The author is grateful to Charles Weibel for his valuable comments and various suggestions during the preparation of this article. He would also like to thank Omprokash Das for some useful discussions. Further, he would also like to thank the referee for useful comments and suggestions.

\section{preliminaries}\label{basic}

\subsection*{Subintegral and Seminormal extension} A commutative ring extension $A\hookrightarrow B$ is subintegral if $B$ is integral over $A$ and $\rm{Spec}(B) \to \rm{Spec}(A)$ is a bijection inducing isomorphisms on all residue fields. We say that $A\hookrightarrow B$ is seminormal (or $A$ is seminormal in $B$) if there is no subextension $A\subset C \subset B$ with $C\neq A$ and $A\subset C$ is subintegral (equivalently, whenever $b \in B$ and $b^2, b^3 \in A$ then $b\in A).$ More details can be found in \cite{rs}, \cite{swan}.

\subsection*{Relative K-groups} 
Given a map $f : X\to S$ of schemes, $K_{n}(f)= \pi_{n}K(f)$, where $K(f)$ is the homotopy fiber of $K(S)\to K(X).$ 
These relative K-groups fit into the following exact sequence 
 \begin{equation}\label{kgroup}
    \dots \rightarrow K_{n}(f)\rightarrow K_{n}(S)\rightarrow K_{n}(X)\rightarrow K_{n-1}(f)\rightarrow K_{n-1}(S)\rightarrow \dots .
   \end{equation}
   Let $NK_{n}(f)= \ker [K_{n}(f[t])\stackrel{t\mapsto 0}\to K_{n}(f)],$ where $f[t]: X\times \mathbb{A}^{1} \to S \times \mathbb{A}^{1}.$ We have a natural decomposition $K_{n}(f[t])\simeq K_{n}(f) \oplus NK_{n}(f).$ By iterating the operation $N,$ we obtain $K_{n}(f[t_{1}, \dots, t_{r}]])\simeq (1\oplus N)^{r}K_{n}(f).$
   For details see \cite{b}, \cite{wei 1}.
\subsection*{Relative Homotopy K-groups} Let $\Delta^{n}= \rm{Spec} (\mathbb{Z}[t_{0}, t_{1}, \dots, t_{n}]/(t_{0}+ t_{1} +\dots +t_{n} -1))$. Then the $n$-th homotopy K-group of a scheme $X$ is $KH_{n}(X)= \pi_{n} (KH(X))$, where $n \in \mathbb{Z}$ and $KH(X)= \rm{hocolim}_{j}~ K(X \times \Delta^{j})$. For a map of schemes $f: X \to S,$ let $KH(f)$ be the homotopy fiber of $KH(S) \to KH(X).$ In fact, $KH(f)= \rm{hocolim}_{j}~ K(f \times \Delta^{j})$ by Lemma 5.19 of \cite{Thomason}. Then for $n \in \mathbb{Z}$, the n-th relative homotopy K-group of $f$ is $KH_{n}(f)= \pi_{n}(KH(f)).$   The relative homotopy K-groups fit into the following exact sequence $\rm{Seq}(KH_{n}, f)$
 \begin{equation}\label{khgroup}
    \dots \rightarrow KH_{n}(f)\rightarrow KH_{n}(S)\rightarrow KH_{n}(X)\rightarrow KH_{n-1}(f)\rightarrow KH_{n-1}(S)\rightarrow \dots .
   \end{equation}    For more details, we refer to \cite{weiKH} and  Chapter IV.12 of \cite{wei 1}.
   
   \begin{remark}\label{K=KH}\rm{
    For a scheme $X$, there is a  natural map $K(X) \to KH(X).$ Therefore, we get a natural map $K(f) \to KH(f)$ for any map $f: X \to S$ of schemes. In particular, there are natural maps $K_{n}(f)\to KH_{n}(f)$ for all $n$. For every scheme $X$, $KH_{n}(X)\cong KH_{n}(X \times \mathbb{A}^{t})$ for all $n$ and $t\geq 0.$ It is also well known that for a regular scheme $X,$ $K_{n}(X)\cong KH_{n}(X)$ for all $n.$ Using the exact sequences (\ref{kgroup}) and (\ref{khgroup}),} the following facts are easy to check:
    \begin{enumerate}
     \item If $X$ and $S$ are regular schemes then $K_{n}(f)\cong KH_{n}(f)$ for all $n.$ 
     \item (Homotopy Invariance) $KH_{n}(f)\cong KH_{n}(f \times \mathbb{A}^{t})$ for all $n$ and all $t$.
    \end{enumerate} \end{remark}

\section{Relative negative K-theory of smooth, affine maps}\label{using KST}
In this section, we prove Theorem \ref{smooth affine}, which is a vanishing and regularity result for relative negative $K$-groups of a smooth, affine map. To prove this, we need some preparations. Let us begin with the following observation.

\begin{lemma}\label{basic reduction}
 Let $f: X \to S$ be a map of noetherian schemes with $\dim S=d.$ Then the following are true:
 \begin{enumerate}
  \item for $n>d,$ $K_{-n}(X)=0$ if and only if $K_{-n-1}(f)=0.$ 
  \item for $n\geq d,$ $X$ is $K_{-n}$-regular if and only if $f$ is $K_{-n-1}$-regular.
 \end{enumerate}

\end{lemma}
\begin{proof}
 By Theorem B of \cite{KST}, $K_{-n}(S)=0$ for $n>d$ and $S$ is $K_{-n}$-regular for $n\geq d.$ Now the first assertion follows from the long exact sequence (\ref{kgroup}). For the second assertion, apply $N^r$ to the sequence (\ref{kgroup}) and use the fact $S$ is $K_{-n}$-regular for $n\geq d$.
\end{proof}

\begin{lemma}\label{0 dim case}
 Let $f: X \to S$ be a smooth map of noetherian schemes with $\dim S=0.$ Assume that $S$ is reduced. Then $K_{-n}(f)=0$ for $n>1$ and $f$ is $K_{-n}$-regular for $n>0$. 
\end{lemma}

\begin{proof}

 Since $S$ is a noetherian reduced scheme of dimension $0,$ $S= \rm{Spec}(k_{1}) \sqcup \dots \sqcup\rm{Spec}(k_{n}).$ Here each $k_{i}$ is a field. So, we can write $X$ as a finite disjoint union of open subschemes $f^{-1}(\rm{Spec}(k_{i}))$ which are regular. Hence $X$ is regular. Then $X$ is $K_{n}$-regular for all $n$ and $K_{-n}(X)=0$ for $n>0.$ We also have $K_{-n}(S)=0$ for $n>0$ and $S$ is $K_{-n}$-regular for $n\geq 0.$ Therefore by (\ref{kgroup}), $K_{-n}(f)=0$ for $n>1$ and $f$ is $K_{-n}$-regular for $n>0$. 
\end{proof}

Let $\mathbb{K}$ be the presheaf of spectra of non-connective Bass $K$-theory on $S_{Zar}.$   For a morphism of schemes $f: X \to S$, let $\mathbb{K}^{f}$ be the presheaf of spectra on $S_{Zar},$ defined as $$\mathbb{K}^{f}(U)= \rm{hofib}[ \mathbb{K}(U) \to \mathbb{K}(X \times_{S} U)]. $$ 
Similarly, we can define the nil presheaf of spectra $N^{r}\mathbb{K}^{f}$ on $S_{Zar}$ for $r>0.$

Let $F$ be a presheaf of spectra on a scheme $X.$ We say that $F$ has the Mayer-Vietoris property ( for the Zariski topology on $X$) if for every pair of open subschemes $U$ and $V$ the following square is homotopy cartesian:$$\begin{CD}
     F(U\cup V) @>>> F(U)\\
    @VVV  @VVV \\                 
    F(V) @>>> F(U\cap V)
  \end{CD}$$

We say that $F$ satisfies Zariski descent if it has the Mayer-Vietoris property. For more details, we refer to Chapter V.10 of \cite{wei 1}. 

\begin{lemma}\label{zariskidescent}
 $\mathbb{K}^{f}$ and $N^{r}\mathbb{K}^{f}$ satisfy Zariski descent.
\end{lemma}

\begin{proof}
  We have a sequence of presheaves of spectra on $S$, 
 $$\mathbb{K}^{f} \to \mathbb{K} \to f_{*}\mathbb{K}.$$
  It is easy to check that if $\mathbb{K}$ satisfies Zariski descent then $f_{*}\mathbb{K}$ does too. By Corollary V.7.10 of \cite{wei 1}, $\mathbb{K}$ satisfies Zariski descent. Then $\mathbb{K}^{f}$ satisfies Zariski descent (see  Exercise V.10.1 of \cite{wei 1}). By a similar argument, $N^{r}\mathbb{K}^{f}$ satisfies Zariski descent. \end{proof}

For a morphism of schemes $f: X \to S$, let $\mathcal{K}_{n,zar}^{f}$ be the Zariski sheafification of the presheaf $U\mapsto K_{n}(f|_{U})$. Here $f|_{U}$ is the map $f^{-1}(U) \to U.$

\begin{lemma}\label{reduction to reduced base}
 Let $f: X\to S$ be an affine map of noetherian schemes with $\dim S=d.$ Then the canonical map $K_{-n}(f) \to K_{-n}(f \times_{S} S_{red})$ is an isomorphism for $n>d,$ where $f \times_{S} S_{red}: X\times_{S} S_{red} \to S_{red}.$
\end{lemma}

\begin{proof}
 Write $\tilde f$ for $f \times_{S} S_{red}.$ First we suppose that $X= \rm{Spec}(B)$ and $S= \rm{Spec}(A).$ Then $X \times_{S} S_{red}= \rm{Spec}(B \otimes_{A}A/nil(A)).$ Note that $\rm{nil}(A)B$ is a nil ideal of $B$. Then by comparing sequences (see (\ref{kgroup}))  for $ f$ and $\tilde f$, we get $K_{-n}(f)\cong K_{-n}(\tilde f)$ for $n>0$ because for any ring $R,$ $K_{-n}(R)\cong K_{-n}(R/I)$ for $n\geq 0$ with $I$ a nil ideal. Now by looking at the stalk level it is easy to see that $\mathcal{K}_{n}^{f}\cong \mathcal{K}_{n}^{\tilde f}$ for all $n<0$ as a Zariski sheaf on $S.$ There is a canonical map of Zariski descent spectral sequences  (Theorem 10.3 of \cite{TT}),
 
 $$E_{2}^{p. q}= H^{p}(S, \mathcal{K}_{-q, zar}^{f})\Rightarrow K_{-p-q}(f)$$ to 
 $$E_{2}^{p. q}= H^{p}(S_{red}, \mathcal{K}_{-q, zar}^{\tilde f})\Rightarrow K_{-p-q}(\tilde f),$$ which is an isomorphism on $E_{2}^{p, q}$ page for $q> 0.$ Moreover, the Zariski cohomological dimension of $S$ is at most $d$. Hence the result. \end{proof}

\begin{remark}\label{may fail for nonaffine}\rm{
 Lemma \ref{reduction to reduced base} may fail if $f$ is not an affine map. Consider $X= C \times_{k} S,$ where $C$ is a smooth, projective curve of genus $g>0$ over a field $k$ and $S= \rm{Spec} (k[s]/s^2).$ Since $K_{-n}(S) \cong K_{-n}(S_{red})$ for $n\geq 0,$ we get $K_{-n}(f)\cong K_{-n}(f_{red})$ for $n>0$ if and only if $K_{-n}(X)\cong K_{-n}(X\times_S S_{red})$ for $n\geq 0.$ In this case $X\times_S S_{red}=C$ and  $\Pic(X) \ncong \Pic(X\times_S S_{red})= \Pic (C).$  Indeed, $\mathcal{O}_{X}^{\times}\cong \mathcal{O}_{C}^{\times} \times \mathcal{O}_{C}$ as a sheaf of abelian groups and thus $\Pic(X)\cong \Pic(C) \times H^{1}(C,\mathcal{O}_{C}).$ Here $H^{1}(C,\mathcal{O}_{C})$ is a $g$-dimensional $k$-vector space. Therefore, $K_{0}(X) \ncong K_{0}(X\times_S S_{red})= K_{0}(C).$} 
\end{remark}

\begin{lemma}\label{reduction to affine}
 Let $f: X \to S$ be a map of noetherian schemes. Suppose $\dim S=d.$ Write $f_{s}$ for the map $X \times_S \mathcal{O}_{S,s} \to \mathcal{O}_{S,s},$ $s\in S.$ Then the following are true: 
 \begin{enumerate}
  \item If $K_{-n}(f_{s})=0$ for all $s\in S$ with $n>\dim \rm{Spec}(\mathcal{O}_{S,s}) +1$ then $K_{-n}(f)=0$ for $n> d+1.$ 
  \item If $N^{r}K_{-n}(f_{s})=0$ for all $s\in S$ with $n> \dim \rm{Spec}(\mathcal{O}_{S,s})$ and $r>0$ then $N^{r}K_{-n}(f)=0$ for $n>d$ and $r>0.$
 \end{enumerate}
 \end{lemma}

\begin{proof}
 The result is clear by Lemma \ref{zariskidescent} and Proposition 6.1 of \cite{KST}. More precisely, apply Proposition 6.1 of \cite{KST} to the presheaves of spectra $\mathbb{K}^{f}[-1]$ and $N^{r}\mathbb{K}^{f}.$
\end{proof}

We are now ready to prove Theorem \ref{smooth affine}.\medskip

{\it Proof of Theorem \ref{smooth affine}:}
 By Lemma \ref{reduction to affine}, we can assume that $S$ is affine. We can also assume that $S$ is reduced by Lemma \ref{reduction to reduced base}. We prove by using induction on $\dim S$. If $\dim S=0$ then the assertion is true by Lemma \ref{0 dim case}. Suppose $d>0.$ Assume that for every smooth, affine map $X \to S$ with $\dim S<d,$ we have $K_{-n}(X)=0$ for $n>\dim S$ (see Lemma \ref{basic reduction}). Let $n< -d$ and consider an element  $\xi$ in $K_{n}(X).$ Here $f$ is smooth and quasi-projective. Apply Proposition 5 of \cite{KS} to the map $f: X \to S.$ Then there exist a projective birational map $p: S^{'} \to S$ such that $\tilde p^{*}\xi=0$ where $\tilde p: X^{'}= X\times_S S^{'} \to X.$ We can choose a nowhere dense closed subset $Y\hookrightarrow S$ such that $p$ is an isomorphism outside $Y$. Then we obtain the following
 abstract blow-up squares 
 
 \begin{equation*}
  \begin{CD}
     Y^{'} @>>> S^{'}\\
    @VVV  @V p VV \\                 
    Y @>>> S
  \end{CD}
  \qquad \rm{and}
  \begin{CD}
   X\times_S Y^{'} @>>> X^{'}\\
    @VVV  @V \tilde p VV \\                 
    X\times_S Y @>>> X
  \end{CD}
\end{equation*}

   By applying Theorem A of \cite{KST}, we get a long exact sequence 
   $$ \dots  \to   \prolim{l}  K_{n+1}(X\times_S Y_{l}^{'}) \to K_{n}(X) \to  \prolim{l} K_{n}(X\times_S Y_{l})\oplus K_{n}(X^{'}) \to \dots $$ of pro-groups. Here $Y_{l}$ (resp. $Y_{l}^{'}$) is the $l$-th infinitesimal thickening of $Y$ (resp. $Y^{'}$) in $S$ (resp. $S^{'}$).   Observe that $X\times_S Y \to Y$ and $X\times_S Y^{'} \to Y$ are smooth, affine with $\dim Y< d$ and $\dim Y^{'}<d.$ Then by the induction hypothesis on $d$, the pro-groups involving $K_{n+1}(X \times_S Y_{l}^{'})$ and $K_{n}(X \times_S Y_{l})$ vanish. Therefore, $\tilde p^{*}: K_{n}(X) \to K_{n}(X^{'})$ is injective and hence $\xi =0.$ This proves the first part.
   
   In the second part, we can assume that $S$ is affine by Lemma \ref{reduction to affine}. Then $X$ is affine.  Now by the proof of Lemma \ref{reduction to reduced base}, we have $K_{-n}(f)\simeq K_{-n}(f\times_S S_{red})$ and $K_{-n}(f\times \mathbb{A}^r)\simeq K_{-n}((f\times_S S_{red})\times \mathbb{A}^r)$ for $n>0$ because $X$ and $S$ are affine. Therefore, we can also assume that $S$ is reduced. Again, we use induction on the dimension of $S.$ If $\dim S=0$ then the assertion is true by Lemma \ref{0 dim case}. Suppose $d>0.$ Assume that for every smooth, affine map $X \to S$ with $\dim S<d$, we have $N^{r}K_{-n}(X)=0$ for  $n\geq \dim S$ and $r>0$ (see Lemma \ref{basic reduction}). Let $n \leq -d$ and $r>0$. For each $r,$ we can argue the inductive step separately.  Consider $\xi \in K_{n}(\mathbb{A}_{X}^{r}).$ Apply Proposition 5 of \cite{KS}, to the map $\mathbb{A}_{X}^{r} \to \mathbb{A}_{S}^{r} \to S$, which is smooth and quasi-projective. Then there exist a projective birational map $p: S^{'} \to S$ such that $\tilde p^{*}\xi=0$ where $\tilde p: \mathbb{A}_{X^{'}}^{r} \to \mathbb{A}_{X}^{r}$ and $X^{'}= X\times_S S^{'}.$ We can choose a nowhere dense closed subset $Y\hookrightarrow S$ 
such that $p$ is an isomorphism outside $Y$. Now we have the following commutative diagram

   \scriptsize
   $$\begin{CD}
     @>>> \prolim {l} K_{n+1}(X\times_S Y_{l}^{'}) @>>> K_{n}(X) @>>> \prolim{l} K_{n}(X\times_S Y_{l})\oplus K_{n}(X^{'}) @>>> \dots \\
    @.   @V \beta_{Y^{'}}^{*} VV   @V \beta^{*} VV     @V\beta_{Y}^{*} \oplus \beta_{X^{'}}^{*} VV  \\
     @>>> \prolim{l} K_{n+1}(\mathbb{A}_{X\times_S Y_{l}^{'}}^{r}) @>>> K_{n}(\mathbb{A}_{X}^{r}) @>>> \prolim{l} K_{n}(\mathbb{A}_{X\times_S Y_{l}}^{r})\oplus K_{n}(\mathbb{A}_{X^{'}}^{r}) @>>> \dots,
   \end{CD}$$ \normalsize
   
   where the horizontal sequences are exact by Theorem A of \cite{KST}. Here $\beta$ is the projection map $ \mathbb{A}_{X}^{r} \to X$ and $\beta^{*}$ is the induced morphism. Since $\dim Y < d$ and $\dim Y^{'}< d$, the left and right column maps $\beta_{Y^{'}}^{*},$ $\beta_{Y}^{*}$ are isomorphism by the induction hypothesis. By the first part the pro-group in the upper horizontal sequence involving $K_{n}(X \times_S Y_{l})$ vanishes. Now a simple diagram chase gives that $\beta^{*}$ is surjective. Note that $\beta^{*}$ is always injective because $\beta$ always has a section. Hence we get the result. \qed

 \begin{remark}\label{reduced base ok}\rm{
    In the proof of the Theorem \ref{smooth affine}, the affineness of $f$ is only used to make the reduction that the base scheme $S$ is reduced. Therefore, if we assume that $S$ is a $d$-dimensional reduced scheme then the statement of Theorem \ref{smooth affine} is also true for any smooth, quasi-projective map $f: X\to S.$ }
   \end{remark}
   
   The above remark allows us to write the following
   
   \begin{theorem}\label{smooth quasiprojective}
    Let $f: X \to S$ be a smooth, quasi-projective map of noetherian schemes with $S$ reduced. Assume that $\dim S= d.$ Then $K_{-n}(f)=0$ for $n>d+1$ and $f$ is $K_{-n}$-regular for $n>d.$
   \end{theorem}

   \begin{remark}\label{regularity fails in general}\rm{
    The regularity result may fail if $f$ is smooth, quasi-projective with non-reduced base. Consider $X= C \times_{k} S,$ where $C$ is a smooth, projective curve of genus $g>0$ over the field $\mathbb{Q}$ and $S= \rm{Spec} (\mathbb{Q}[s]/s^2).$ Since $\dim(X)=1$, $NK_{-q}(X)=0$ for $q\geq 1.$ We have a Zariski descent spectral sequence for $X$
     $$H^{p}(X, \mathcal{NK}_{-q}) \Rightarrow NK_{-p-q}(X).$$ For an open affine subset $U$ of $C,$ $NK_{-q}(U \times_{\mathbb{Q}} S)= NK_{-q}(U)=0$ for $q\geq 0.$ For a regular $\mathbb{Q}$-algebra $A$, $K_{1}(A[t]/t^2, (t))= \ker [K_{1}(A[t]/t^2) \to K_{1}(A)]$ is isomorphic to $A$ as an abelian group and $NK_{1}(A[t]/t^2)$ is just an infinite copies of $A.$ So, we get \[
\mathcal{NK}_{-q}= \begin{cases} 0 & {\rm if}~ q\geq0 \\
                    \oplus_{i} \mathcal{O}_{C} & {\rm if}~ q=-1.
  \end{cases}
\] Now, by the spectral sequence $NK_{0}(X)= H^{1}(X, \mathcal{NK}_{1})\simeq \oplus_{i} H^{1}(C, \mathcal{O}_{C})\neq 0.$ }
   \end{remark}

   However, we have the following observation in the vanishing part for a smooth, quasi-projective map with non-reduced base.
   
   \begin{proposition}\label{nonreduced base}
 Let $f: X \to S$ be a smooth, quasi-projective map of schemes with $\dim S=d.$ Assume that $\dim X= d+1.$ Then $K_{-n}(f)=0$ for $n> d+1.$ 
\end{proposition}

\begin{proof}
 By Lemma \ref{basic reduction}, it is enough to show that $K_{-n}(X)=0$ for $n> d.$ By applying the Theorem \ref{smooth quasiprojective} for the map $f \times_{S} S_{red}: X\times_{S} S_{red} \to S_{red},$ we get $K_{-n}(X\times_{S} S_{red})=0$ for $n>d.$ Note that $(X\times_{S} S_{red})_{red}= X_{red}.$ Then $$K_{-n}(X\times_{S} S_{red})\cong K_{-n}((X\times_{S} S_{red})_{red})= K_{-n}(X_{red})\cong K_{-n}(X)$$ for $n>d,$ where the first and last isomorphisms by Lemma 4.6 of \cite{Morrow}. Hence the result. \end{proof}
 
The next remark was pointed out to me by C. Weibel.
\begin{remark}\label{false in surface case}\rm{
 The above proposition may fail if $\dim(X)= d+2.$ Consider $X= D\times_{\mathbb{C}} S$, where $D$ is a smooth, projective surface over $\mathbb{C}$ and $S= \rm{Spec} (\mathbb{C}[s]/s^2).$ Assume that the geometric genus of $D,$ $p_g(D)= \dim_{\mathbb{C}} H^2(D, \mathcal{O}_{D})>0.$ In particular, we can take $D$ to be $K3$ surface because the geometric genus is $1$ in this case. Note that $\dim(X)=2$ and $K_{-n}(X)\cong K_{-n}(X_{red})= K_{-n}(D)=0$ for $n>1.$ Now, we have a Zariski descent spectral sequence for $X$
 $$H^{p}(X, \mathcal{K}_{-q}) \Rightarrow K_{-p-q}(X).$$ For an open affine subset $U$ of $D,$ $K_{-q}(U \times_{\mathbb{C}} S)= K_{-q}(U)=0$ for $q> 0.$ So, we get
 
 \[
\mathcal{K}_{-q}= \begin{cases} 0 & {\rm if}~ q>0 \\
                      \mathbb{Z}& {\rm if}~ q=0\\
                      \mathcal{O}_{D}^{\times} \times \mathcal{O}_{D} & {\rm if}~ q=-1.
  \end{cases}
\]
 Moreover, the Zariski cohomological dimension of $X$ is $2$ and $H^{p}(X, \mathbb{Z})=0$ for $p>0.$  Therefore by the spectral sequence we get $K_{-1}(X)= H^{2}(X, \mathcal{K}_{1})= H^2(D, \mathcal{O}_{D})\neq 0.$}
\end{remark}

\section{Relative negative homotopy K-theory of smooth maps}\label{KH vanish}

In this section, all the schemes are defined over a field $k.$ 
The main goal is to prove Theorem \ref{Smooth,surjective}, which is a vanishing result for relative negative homotopy K-groups of a smooth map.

Let $X$ be a scheme over a field $k$ with the singular locus $Z= {\rm Sing} (X_{red}).$ A resolution of singularities (ROS) of $X$ is a proper map $p: \tilde {X} \to X,$ where $\tilde {X}$ is smooth over $k,$ and $p$ restricts to an isomorphism $ p^{-1} ((X - Z)_{red})\stackrel{\sim}\to (X - Z)_{red}.$  We say that the resolution of singularities holds over $k$ if ROS exists for every scheme $X$ over $k$.

If the resolution of singularities holds over $k$ then we can always produce a smooth cdh-cover of $X$ in the following way: First start with a resolution $\tilde {X} \to X.$ Then $Z\sqcup \tilde{X} \to X$ is a cdh-cover of $X.$ In the next step, we can resolve the singularities of $Z$ if it is singular. By iterating this process we get a smooth cdh-cover of $X$ and this process must terminate after finite steps because the dimension of the singular set decreases each step.  


Recall that $\mathcal{K}_{n, cdh}^{f}$ is the cdh-sheafification of the relative K-theory presheaf (see the introduction). By replacing $K$ by $KH,$ we get $\mathcal{KH}_{n, cdh}^{f}.$ 

\begin{lemma}\label{K=KH as a cdh}
 Let $f: X \to S$ be a smooth map of schemes over a field $k$. Assume that the resolution of singularities holds over $k$. Then $\mathcal{KH}_{n, cdh}^{f}\cong\mathcal{K}_{n, cdh}^{f}$ for all $n$ as a sheaf on the cdh-site over $S.$ Moreover, $\mathcal{K}_{n, cdh}^{f}$ is zero for $n<-1.$ 
\end{lemma}

\begin{proof}
Since the  resolution of singularities holds over $k,$ schemes are locally smooth in the cdh topology. So, we may assume that $S$ is smooth over $k.$ Moreover, it is well known that the negative absolute K-theory of regular schemes vanish. Hence the assertion follows by the exact sequence (\ref{kgroup}) and Remark \ref{K=KH}.\end{proof}


Let $\mathbb{KH}$ be the presheaf of spectra of homotopy $K$-theory on $Sch/S.$  For a morphism of schemes $f: X \to S$, let $\mathbb{KH}^{f}$ be the presheaf of spectra on the category $Sch/S,$ defined as $$\mathbb{KH}^{f}(U)= \rm{hofib}[ \mathbb{KH}(U) \to \mathbb{KH}(X \times_{S} U)].$$ 

{\it Proof of Theorem \ref{Smooth,surjective}:}
 In \cite[Theorem 3.9]{Cisinki},  Cisinski proved that $\mathbb{KH}$ satisfies cdh descent when $S$ is a finite dimensional noetherian scheme (it is proved in \cite{Has} when $S$ is a point). Then  $\mathbb{KH}^{f}$ satisfies cdh descent (the arguments are similar to Lemma \ref{zariskidescent}). Therefore, we have a descent spectral sequence (by Theorem 3.4 of \cite{CHSW} and Theorem 2.8 of \cite{Has})  $$
E_{2}^{p,q}= H_{cdh}^{p}(S, \mathcal{KH}_{-q, cdh}^{f})\Rightarrow KH_{-p-q}(f).
$$ Also by Lemma \ref{K=KH as a cdh}, $\mathcal{KH}_{q, cdh}^{f}\cong\mathcal{K}_{q, cdh}^{f}$ as a sheaf on the cdh-site over $S$ and for $q<-1,$ $\mathcal{K}_{q, cdh}^{f}$ is zero. Moreover, the cdh cohomological dimension is at most $d$ (see Theorem 5.13 of \cite{SV}). Hence, we get $KH_{-n}(f)=0$ for $n>d+1$ and $H_{cdh}^{d}(S, \mathcal{K}_{-1, cdh}^{f})= KH_{-d-1}(f).$ \qed

\section{Relative negative K-theory of subintegral maps}\label{subint case}
In this section, we study the relative negative $K$-groups of a subintegral map of schemes. In particular, we prove Theorem \ref{K-dim}. We begin with the following definition.

\begin{definition}
  Let $f: X\to S$ be a faithful affine morphism, i.e, $f$ is affine and the structure map $\mathcal{O}_{S}\to f_{*}\mathcal{O}_{X}$ is injective. We say that $f$ is subintegral if $\mathcal{O}_{S}(U)\to f_{*}\mathcal{O}_{X}(U)$ is subintegral for all affine open subsets $U$ of $S$. 
\end{definition}

Next, we recall a notion of relative Picard group $\Pic (f)$ for a map $f: X\to S$ of schemes. The relative $\Pic (f)$ is the abelian group  generated by $[L_1,\alpha,L_2]$, where the $L_i$ are 
line bundles on $S$ and $\alpha:f^*L_1\to f^*L_2$ is an isomorphism.
The relations are:
\begin{enumerate}
\item $[L_1,\alpha,L_2] + [L'_1,\alpha',L'_2] = 
[L_1\otimes L'_1,\alpha\otimes\alpha',L_2\otimes L'_2]$;
\item $[L_1,\alpha,L_2] + [L_2,\beta,L_3] = [L_1,\beta\alpha,L_3]$;
\item $[L_1,\alpha,L_2]=0$ if $\alpha=f^*(\alpha_0)$ 
for some $\alpha_0:L_1\cong L_2$.
\end{enumerate}

This relative Picard group $\Pic (f)$ fits into the following exact sequence 

\begin{align}\label{U-Pic}
 \mathcal{O}^{\times}(S)\to \mathcal{O}^{\times}(X) \to \Pic(f) \to \Pic(S) \to \Pic(X).
\end{align} Some relevant details and basic properties can be found in \cite{Bass-Tata}, \cite{SW}, \cite{SW1}.

 Given a ring extension $A\hookrightarrow B,$ let $\mathcal{I}(A, B)$( or $\mathcal{I}(f)$) denote the multiplicative group of invertible $A$-submodules of $B.$  We refer to section 2 of \cite{rs} for details. There is a natural group homomorphism $\psi: \mathcal{I}(f) \to \Pic(f), I\mapsto [I, \alpha, A]$, where $\alpha: I\otimes_{A} B\to B$ is the canonical isomorphism sending $i\otimes b$ to $i\cdot  b$ for $i\in I, b\in B.$
 
 If $\mathbb{Q}\subset A$ then a natural group homomorphism $\xi_{B/A}: B/A \to \mathcal{I}(f)$ is constructed by Roberts and Singh in \cite{rs}. 

\begin{lemma}\label{ring}
 Let  $f:A\hookrightarrow B$ be a subintegral extension of $\mathbb{Q}$-algebras. Then the natural map $\psi \circ\xi_{B/A}: B/A \to \Pic(f)$ is an isomorphism.
\end{lemma}
\begin{proof}
 By Lemma 1.2 of \cite{SW1}, $\psi$ is an isomorphism. Moreover, $\xi_{B/A}$ is an isomorphism follows from Theorem 5.6 of \cite{rs} and Theorem 2.3 of \cite{rrs}. Hence the lemma.
\end{proof}

The following Proposition generalizes the above result for schemes.

\begin{proposition}\label{scheme}
 Let $f: X\to S$ be a subintegral morphism of $\mathbb{Q}$-schemes. Then 
 $\Pic(f)\cong H_{zar}^{0}(S, \it{f_{*}\mathcal{O}_{X}/\mathcal{O}_{S}})$.
\end{proposition}
\begin{proof}
 Let $s\in S$. Then $(f_{*}\mathcal{O}_{X}/\mathcal{O}_{S})_{s}\cong B_s/A_s\cong\Pic (f_s)\cong (f_{*}\mathcal{O}^{\times}_X/\mathcal{O}^{\times}_{S})_s$, where the second isomorphism by Lemma \ref{ring} and the third isomorphism by the exact sequence (\ref{U-Pic}). This implies that $f_{*}\mathcal{O}_{X}/\mathcal{O}_{S}\cong f_{*}\mathcal{O}^{\times}_X/\mathcal{O}^{\times}_{S}$ as a sheaves on $S$. Now the result follows from Lemma 5.4 of \cite{SW}.\end{proof}

\begin{proposition}\label{quasi-coherent}
 Let $f: X\to S$ be a subintegral morphism of noetherian $\mathbb{Q}$-schemes. Then the following are true:
 \begin{enumerate}
  \item $f_{*}\mathcal{O}^{\times}_X/\mathcal{O}^{\times}_{S}$ is the sheaf of abelian groups underlying a quasi-coherent sheaf on $S.$
  \item $H_{zar}^{i}(S, f_{*}\mathcal{O}^{\times}_X/\mathcal{O}^{\times}_{S}) =H_{et}^{i}(S, f_{*}\mathcal{O}^{\times}_X/\mathcal{O}^{\times}_{S})$.
  \item If $S$ is affine and $f$ is finite then for $i>1$, $H_{\tau}^{i}(S, \mathcal{O}^{\times}_{S})\cong H_{\tau}^{i}(X, \mathcal{O}^{\times}_{X})$, where $\tau \in \{ zar, et\}.$
 \end{enumerate}
\end{proposition}
\begin{proof}
 (1) Since $f$ is affine, $f_{*}\mathcal{O}_X$ is quasi-coherent. Then the quotient $f_{*}\mathcal{O}_X/\mathcal{O}_S$ is also  quasi-coherent. Therefore, we get the result by using the fact that $f_{*}\mathcal{O}_{X}/\mathcal{O}_{S}\cong f_{*}\mathcal{O}^{\times}_X/\mathcal{O}^{\times}_{S}$ (see the proof of Proposition \ref{scheme}). 
 
 (2) This follows from the fact that Zariski and \'etale cohomology coincides for a quasi-coherent sheaf (see Remark 3.8 of \cite{Milne}).
 
 (3) Consider the long exact cohomology sequence associated to the following exact sequence of sheaves on $S,$
 $$ 1 \to \mathcal{O}_{S}^{\times} \to f_{*}\mathcal{O}_{X}^{\times} \to f_{*}\mathcal{O}^{\times}_X/\mathcal{O}^{\times}_{S} \to 1 .$$ Since $f$ is finite, $H_{\tau}^{i}(S, f_{*}\mathcal{O}^{\times}_{X})\cong H_{\tau}^{i}(X, \mathcal{O}^{\times}_{X}).$ By (1),  $H_{zar}^{i}(S, f_{*}\mathcal{O}^{\times}_X/\mathcal{O}^{\times}_{S})=0$ for $i>0$, because $S$ is affine. Hence the assertion. \end{proof}

\begin{remark}\rm{
 The statement (3) of the above Proposition may fail for $i=1$. For example, consider $A=\mathbb{Q}[t^2, t^3]$ and $B= \mathbb{Q}[t].$ In this case, $\Pic(A)\cong \mathbb{Q}$ and $\Pic(B)=0$.}
\end{remark}

Recall that $\mathcal{K}_{n,zar}^{f}$ is the Zariski sheafification of the presheaf $U\mapsto K_{n}(f|_{U})$. Here $f|_{U}$ is the map $f^{-1}(U) \to U.$ 

\begin{lemma}\label{K=0}
 If $f$ is subintegral then $\mathcal{K}_{-q, zar}^{f}=0$ for $q>0$.
\end{lemma}
\begin{proof}
 Each stalk of  $\mathcal{K}_{-q, zar}^{f}$ is $K_{-q}(A, B)$, where $A\subset B$ is a subintegral extension of local rings. By Proposition 2.5 of \cite{SW1}, $K_{-q}(A, B)=0$ for $q>0$. Hence the result. \end{proof}

{\it Proof of Theorem \ref{K-dim}:}
We have a descent spectral sequence 
 $$
E_{2}^{p,q}= H_{zar}^{p}(S, \mathcal{K}_{-q, zar}^{f})\Rightarrow K_{-p-q}(f).
$$ By Lemma \ref{K=0}, $\mathcal{K}_{-q, zar}^{f}=0$ for $q>0$. Moreover, $\mathcal{K}_{0, zar}^{f}\cong f_{*}\mathcal{O}^{\times}_X/\mathcal{O}^{\times}_{S}$ by the exact sequence (2.1) of \cite{SW1}. Since the Zariski cohomological dimension of $S$ is at most d, we get the first two assertions. The last assertion follows from Proposition \ref{quasi-coherent}(2). \qed

Let $X$ be a scheme. The seminormalization of $X$ can be obtained by mimicking the normalization process. For each affine open subset $U= \rm{Spec}(A)$ of $X$, let $^{^+}\!\!\!A $ be the subintegral closure (or seminormalization) of $A$ in its total quotient ring. Let $^{^+}\!\!\! ~U= \rm{Spec}(^{^+}\!\!\!A ).$ Now by gluing together such schemes $ ^{^+}\!\!\!~ U$ we get $ X_{\rm{sn}}$, which we call the seminormalization of $X$. Clearly, then $ X_{\rm{sn}} \to X$ is a subintegral morphism.

\begin{corollary}\label{seminormalization}
 Let $X$ be a $d$-dimensional noetherian scheme. Let $X_{\rm{sn}}$ be the seminormalization of $X$. Then $K_{-n}(X)\cong K_{-n}(X_{\rm{sn}})$ for $n> d$ and $K_{-d}(X) \to K_{-d}(X_{\rm{sn}})$ is surjective.
\end{corollary}

\begin{proof}
 Clear from Theorem \ref{K-dim}(1) and the long exact sequence (\ref{kgroup}).
\end{proof}

\section{On Regularity}\label{relative regularity}

A theorem of Vorst says that if a ring $A$ is $K_{n}$-regular then it is $K_{n-1}$-regular (see V. 8.6 of \cite{wei 1}). Now we prove Theorem \ref{Analouge of Vorst}, which is a relative version of Vorst's result. \medskip

For a ring $A$ and an element $a\in A,$ consider the ring homomorphism $[a]: A[X] \to A[X],$ $X\mapsto aX.$ In \cite{Vorst}, Vorst defined $NK_{n}(A)_{[a]}$ as the direct limit of the directed system
$$ NK_{n}(A)\stackrel{[a]}\to NK_{n}(A) \stackrel{[a]}\to NK_{n}(A) \stackrel{[a]}\to \dots    .$$ For details, we refer to section 1 of \cite{Vorst}.  By applying $N$ to the exact sequence (\ref{kgroup}), we get the following exact sequence 
 \begin{equation}\label{nkgroup}
  \dots \rightarrow NK_{n+1}(B)  \rightarrow NK_{n}(f)\rightarrow NK_{n}(A)\rightarrow NK_{n}(B)\rightarrow NK_{n-1}(f)\rightarrow \dots ~~.
\end{equation} 
 {\it Proof of Theorem \ref{Analouge of Vorst}:}
First we suppose that $K_{n}(f)\cong K_{n}(f[s,t, u]).$ Then $NK_{n}(f[s, t])=0.$ The goal is to show that $NK_{n-1}(f)=0.$ By (\ref{nkgroup}), it is enough to show that  $NK_{n}(A)\rightarrow NK_{n}(B)\to 0$ and $0\to NK_{n-1}(A) \to NK_{n-1}(B).$ 
    Since $NK_{n}(f[s])=NK_{n}(f[s, t])=0,$ we get 
    \begin{equation}\label{nkgroupf[s]}
     NK_{n+1}(A[s])\rightarrow NK_{n+1}(B[s]) \to 0,~~  0\to NK_{n}(A[s])\rightarrow NK_{n}(B[s])
    \end{equation} and 
    
\begin{equation}\label{nkgroupf[s,t]}
     NK_{n+1}(A[s,t])\rightarrow NK_{n+1}(B[s,t]) \to 0,~~   0\to NK_{n}(A[s,t])\rightarrow NK_{n}(B[s,t]).
    \end{equation}
    By Lemma V.8.5 of \cite{wei 1}, $NK_{n}(A[s])_{[s]}\cong NK_{n}(A[s, 1/s]).$ Since the direct limit is an exact functor, from (\ref{nkgroupf[s]})  we obtain $$NK_{n+1}(A[s, 1/s])\rightarrow NK_{n+1}(B[s, 1/s]) \to 0 $$ and $$ 0\to NK_{n}(A[s, 1/s])\rightarrow NK_{n}(B[s, 1/s]).$$ Now applying the Bass fundamental theorem (see Theorem V.8.2 of \cite{wei 1}), we get $$NK_{n}(A)\rightarrow NK_{n}(B)\to 0.$$ In fact, $NK_{n}(A)\cong NK_{n}(B)$ because $NK_{n}(f)=0.$ From (\ref{nkgroupf[s,t]}), we can show that $N^{2}K_{n}(A)\cong N^{2}K_{n}(B)$ by a similar argument. Thus the injection $  0\to NK_{n}(A[s, 1/s])\rightarrow NK_{n}(B[s, 1/s])$  implies that $0\to NK_{n-1}(A) \to NK_{n-1}(B).$  Hence $NK_{n-1}(f)=0.$

   Similarly, by assuming $NK_{n}(f[s, t, u])=0$ and using the Bass fundamental theorem,  we can show that  $N^{3}K_{n}(A)\cong N^{3}K_{n}(B)$ and $N^{2}K_{n-1}(f)=0.$
   
   Therefore, by repeating the same arguments we get $N^{i}K_{n-1}(f)=0$ for all $i>0$. \qed

\begin{proposition}\label{reg for subintegral}
 If $f$ is a non-trivial subintegral map of affine schemes then $f$ cannot be $K_{n}$-regular for $n\geq 0.$ 
\end{proposition}

\begin{proof}
 Since $f$ is subintegral, $K_{0}(f)\cong \Pic(f)$ by Proposition 2.5 of \cite{SW1} and hence $NK_{0}(f)\cong N\Pic(f).$ By Theorem 1.5 of \cite{ss}, $N\Pic(f)=0$ if and only if $f$ is seminormal. Therefore, $NK_{0}(f)\neq 0$ and hence $f$ is not $K_{0}$-regular. Now Theorem \ref{Analouge of Vorst} implies that $f$ cannot be $K_{n}$-regular for $n\geq 0.$\end{proof}
 

\begin{remark}\label{converse of vorst}\rm{
 The converse of  Theorem \ref{Analouge of Vorst} does not hold. Consider $f$ to be a subintegral map of affine schemes. Then $f[t_{1}, t_{2}, \dots, t_{l}]$ also subintegral. Now by Proposition 2.5 of \cite{SW1}, we have  $K_{n}(f)= K_{n}(f[t_{1}, t_{2}, \dots, t_{l}])=0$ for $n<0.$ Hence $f$ is $K_{n}$-regular for $n<0.$ But $f$ is not $K_{0}$-regular by Proposition \ref{reg for subintegral}. In particular, consider $f: \rm{Spec} (\mathbb{Q}[x]/(x^2)) \to \rm{Spec} (\mathbb{Q}).$ Here $f$ is subintegral, $K_{-1}$-regular but not $K_{0}$-regular.} 
\end{remark}

\end{document}